\documentclass[english]{amsproc}
\usepackage[T2A]{fontenc}
\usepackage[english]{babel}
\usepackage{graphics}
\usepackage{amsfonts, amssymb, amscd, amsmath}
\usepackage[dvips]{graphicx}
\usepackage{latexsym}
\usepackage[matrix,arrow,curve]{xy}
\usepackage{wrapfig}

\allowdisplaybreaks[1]

\DeclareMathOperator{\id}{id}

\DeclareMathOperator{\Tor}{Tor} \DeclareMathOperator{\sgn}{sgn}
\DeclareMathOperator{\bideg}{bideg} \DeclareMathOperator{\im}{im}
\DeclareMathOperator{\mdeg}{mdeg} \DeclareMathOperator{\wt}{wt}
\DeclareMathOperator{\supp}{supp}

\newcommand{\Zo}{\mathbb{Z}}
\newcommand{\Ro}{\mathbb{R}}
\newcommand{\Co}{\mathbb{C}}
\newcommand{\Noo}{\mathbb{N}}
\newcommand{\ko}{\Bbbk}
\newcommand{\Zt}{\Zo_2}
\newcommand{\eqd}{\stackrel{\text{\tiny def}}{=}}
\newcommand{\congd}{\stackrel{\text{\tiny def}}{\cong}}
\newcommand{\Sn}{S_9}

\newcommand{\res}[1]{{\widetilde{#1}}}

\newcommand{\F}{\mathcal{F}}

\newcommand{\Hr}{\tilde{H}}

\newcommand{\Z}{\mathcal{Z}}
\newcommand{\Zr}{{_{\Ro}\mathcal{Z}}}
\newcommand{\sr}{s_{\Ro}}
\newcommand{\rr}{r_{\Ro}}
\newcommand{\Ur}{{_{\Ro}U}}

\newcommand{\A}{\mathcal{A}}
\newcommand{\R}{\mathcal{R}}
\newcommand{\Cc}{\mathcal{C}}

\newcounter{stmcounter}[section]
\newcounter{thcounter}

\newcounter{defcounter}[section]
\newcounter{problcounter}

\numberwithin{equation}{section}

\theoremstyle{plain}
\newtheorem{cor}[stmcounter]{Corollary}
\newtheorem{stm}[stmcounter]{Statement}
\newtheorem{thm}[thcounter]{Theorem}
\newtheorem{prop}[stmcounter]{Proposition}
\newtheorem{lemma}[stmcounter]{Lemma}
\newtheorem{defin}[stmcounter]{Definition}
\newtheorem{problem}[problcounter]{Problem}
\newtheorem{que}[stmcounter]{Question}
\newtheorem{claim}[stmcounter]{Claim}

\theoremstyle{definition}

\newtheorem{ex}[stmcounter]{Example}
\newtheorem{rem}[stmcounter]{Remark}
\newtheorem{con}[stmcounter]{Construction}

\begin{document}

\title[Buchstaber numbers and classical invariants]{Buchstaber numbers and classical
invariants of simplicial complexes}

\author{Anton Ayzenberg}
\thanks{The author is supported by the JSPS postdoctoral fellowship program}

\address{Osaka City University}
\email{ayzenberga@gmail.com}

\begin{abstract}
Buchstaber invariant is a numerical characteristic of a simplicial
complex, arising from torus actions on moment-angle complexes. In
the paper we study the relation between Buchstaber invariants and
classical invariants of simplicial complexes such as bigraded
Betti numbers and chromatic invariants. The following two
statements are proved. (1) There exists a simplicial complex $U$
such that $s(U)\neq \sr(U)$. (2) There exist two simplicial
complexes with equal bigraded Betti numbers and chromatic numbers,
but different Buchstaber invariants. To prove the first theorem we
define Buchstaber number as a generalized chromatic invariant.
This approach allows to guess the required example. The task then
reduces to a finite enumeration of possibilities which was done
using GAP computational system. To prove the second statement we
use properties of Taylor resolutions of face rings.
\end{abstract}

\maketitle

\section{Introduction}\label{SecIntro}

Let $K$ be a simplicial complex on a set of vertices
$[m]=\{1,2,\ldots,m\}$. In toric topology a special topological
space, called moment-angle complex, is associated to $K$.

\begin{defin}[Moment-angle complex \cite{BP2,BP}]\label{definMA}\mbox{}

(1) Let $D^2\subset\Co$ be the unit disk and $S^1$ --- its
boundary circle. For any simplex $I\in K$ define the subset
$(D^2,S^1)^I\subset (D^2)^m$, $(D^2,S^1)^I =
(D^2)^I\times(S^1)^{[m]\setminus I}$. Here, in the product, disks
stand on the positions from $I$ and circles stand on all other
positions. The moment-angle complex of $K$ is the topological
space
$$
\Z_K = \bigcup_{I\in K}(D^2,S^1)^I \subseteq (D^2)^m.
$$
This subset is preserved by the coordinatewise action of the
compact torus $T^m=(S^1)^m~\curvearrowright~(D^2)^m$, where each
component $S^1$ acts on corresponding $D^2\subset \Co$ by
rotations. This defines the action $T^m\curvearrowright \Z_K$.

(2) Let $D^1=[-1;1]\subset \Ro$ and $S^0=\partial D^1=\{-1,1\}$.
For any simplex $I\in K$ define the subset $(D^1,S^0)^I\subset
(D^1)^m$, $(D^1,S^0)^I = (D^1)^I\times(S^0)^{[m]\setminus I}$. The
\emph{real moment-angle complex} of $K$ is the topological space
$$
\Zr_K = \bigcup_{I\in K}(D^1,S^0)^I \subseteq (D^1)^m.
$$
This subset is preserved by the coordinatewise action of the
finite group $\Zt^m\curvearrowright (D^1)^m$. Here the group
$\Zt=\Zo/2\Zo$ acts on $D^1\subset\Ro$ by changing sign. This
defines the action $\Zt^m\curvearrowright \Zr_K$.
\end{defin}

Constructions in toric topology, in particular moment-angle
complexes, give rise to interesting and nontrivial invariants of
simplicial complexes. Note that the actions $T^m\curvearrowright
\Z_K$ and $\Zt^m\curvearrowright \Zr_K$ are not free if $K$ has at
least one nonempty simplex. The main objects of this paper are
Buchstaber invariants measuring \emph{the degree of symmetry} of
moment-angle complexes.

\begin{defin}[Buchstaber invariant]\label{definBuchNum}\mbox{}

(1) The (ordinary) Buchstaber invariant $s(K)$ of a simplicial
complex $K$ is the maximal rank of toric subgroups $G\subset T^m$
for which the restricted action $G\curvearrowright \Z_K$ is free.

(2) The real Buchstaber invariant $\sr(K)$ is the maximal rank of
subgroups $G\subset \Zt^m$ for which the restricted action
$G\curvearrowright \Zr_K$ is free.
\end{defin}

Here ``rank of subgroup $G\subset\Zt^m$'' means the dimension of
$G$ as a vector subspace over the field of two elements. This
finite field will also be denoted by $\Zt$.


Several approaches to the study of Buchstaber invariants are
developed up to date \cite{Izm1,Izm2,ErArx,ErNew,FM}. We refer to
\cite{ErNewBig} for the comprehensive review of this field. In
this paper we study the connection of Buchstaber invariants with
each other and with other invariants of simplicial complexes.

Generally, there is a bound
\begin{equation}\label{eqBuchGenBound}
1\leqslant s(K)\leqslant \sr(K)\leqslant m-\dim K-1
\end{equation}
In toric topology the case $s(K)=\sr(K)=m-\dim K-1$ is the most
important; it appears quite often. Still there are many examples
of $K$ for which $1<\sr(K)<m-\dim K-1$ or $1<s(K)<m-\dim K-1$. It
is always very difficult to compute $s(K)$ for such examples
(Section \ref{SecRealComplex} contains an example of such
computation). The real invariant $\sr(K)$ is easier because its
calculation allows computer-aided analysis. Thus an important
question is: whether $s(K)=\sr(K)$ for any complex $K$? The answer
is negative.

\begin{thm}\label{thmRealComplex}
There exists a simplicial complex $U$ of dimension $3$ such that
$s(U)~\neq~\sr(U)$.
\end{thm}

\rem Theorem \ref{thmRealComplex} was announced without a proof in
\cite{AyzArx}. The proof was published later in \cite{Ayzs}, but
unfortunately, that issue of the journal was not published in
English. We provide the proof here (Section \ref{SecRealComplex}).

The second block of questions asks about the relation between
Buchstaber invariants and other well-studied invariants. If
$\A(\cdot)$ is an invariant (possibly, a set of invariants) of a
simplicial complex, then the general question is:

\begin{que}\label{queInvProblem} Does $\A(K)=\A(L)$  imply
$s(K)=s(L)$ or $\sr(K)=\sr(L)$?
\end{que}

There are several natural candidates for $\A(\cdot)$:
\begin{itemize}
\item Chromatic number $\gamma(K)$ or its generalizations;
\item $f$-vector (or, equivalently, $h$-vector) of $K$;
\item Topological characteristics of $K$,
e.g. Betti numbers;
\item Topological characteristics of the moment-angle complex $\Z_K$.
\end{itemize}

Classical chromatic number $\gamma(K)$ on itself is too weak
invariant for rigidity question \ref{queInvProblem} to make sense.
On the other hand, Buchstaber invariants can themselves be
considered as generalized chromatic invariants (see Section
\ref{SecApproaches}). N.\,Erokhovets \cite{Er,ErArx} proved that
Buchstaber invariants are not determined by the $f$-vector and the
chromatic number. More precisely, he constructed two simplicial
polytopes, whose boundaries have equal $f$-vectors and chromatic
numbers, but Buchstaber invariants are different.

The cohomology ring of a moment-angle complex is the subject of
intensive study during last fifteen years. It is known
\cite{BP2,Franz} that,
\begin{equation}\label{eqCohomMA}
H^*(\Z_K;\ko)\cong \Tor^{*,*}_{\ko[m]}(\ko[K],\ko) =
\bigoplus\limits_{\ell,j}\Tor^{-\ell,2j}_{\ko[m]}(\ko[K],\ko)
\end{equation}
--- the $\Tor$-algebra of a Stanley--Reisner ring. The
dimensions of graded components
\begin{equation}\label{eqBigrBettiDef}
\beta^{-\ell,2j}(K)\eqd\dim_{\ko}\Tor^{-\ell,2j}_{\ko[m]}(\ko[K],\ko).
\end{equation}
are called bigraded Betti numbers of $K$. In general they depend
on the ground field $\ko$. These invariants represent a lot of
information about $K$ \cite{Stan,BP}. In particular, from bigraded
Betti numbers it is possible to extract: the $h$-vector of $K$;
the ordinary Betti numbers of $K$ and the ordinary Betti numbers
of $\Z_K$ by formulas:
$$
h_0(K)+h_1(K)t+\ldots+h_n(K)t^n=\dfrac{1}{(1-t)^{m-n}}\sum
\beta^{-\ell,2j}(-1)^{\ell} t^j\quad\mbox{\cite[Th.7.15]{BP}};
$$
$$
\dim \Hr^i(K;\ko)=\beta^{-(m-i-1),2m}(K)\quad\mbox{(the part of
Hochster's formula \cite{Hoch},\cite[Th.3.27]{BP})};
$$
$$
\dim
H^i(\Z_K;\ko)=\sum\limits_{-\ell+2j=i}\beta^{-\ell,2j}(K)\quad\mbox{(follows
from \eqref{eqCohomMA})},
$$
where $n=\dim K+1$. Note, that bigraded Betti numbers do not
determine the dimension of $K$. The cone over $K$ always has the
same bigraded Betti numbers as $K$ but the dimension is different.

So far, $\beta^{-i,2j}(K)$ (together with $\dim K$) is a very
strong set of invariants. The question~\ref{queInvProblem} makes
sense for this set of invariants. Still the answer is negative.

\begin{thm}\label{thmBettiBuch}
There exist simplicial complexes $K_1$ and $K_2$ such that
\begin{enumerate}
\item $\beta^{-i,2j}(K_1)=\beta^{-i,2j}(K_2)$ for all $i,j$;
\item $\dim K_1=\dim K_2$;
\item $\gamma(K_1)=\gamma(K_2)$;
\item $s(K_1)\neq s(K_2)$ and $\sr(K_1)\neq\sr(K_2)$.
\end{enumerate}
\end{thm}

We also show that $\Tor$-algebras of the constructed complexes
$K_1$ and $K_2$ have trivial multiplications. Thus not only
bigraded Betti numbers but also multiplicative structure of
$H^*(\Z_K)$ does not determine Buchstaber invariant.

The paper consists of two essential parts which are independent
from each other. Sections~\ref{SecApproaches}
and~\ref{SecRealComplex} form the first part. Theorem
\ref{thmRealComplex} is proved in Section~\ref{SecRealComplex}.
Section~\ref{SecApproaches} clarifies the combinatorial meaning of
Buchstaber invariants and contains definitions and constructions
necessary for understanding the proof. In the second part of the
paper we explore the connection between Buchstaber invariants and
bigraded Betti numbers. This requires some basic homological
algebra and the construction of the Taylor resolution of a
Stanley--Reisner ring. Section~\ref{SecBettiBuch} contains all
necessary definitions and the proof of Theorem~\ref{thmBettiBuch}.

\section{Combinatorial approach to Buchstaber invariants}\label{SecApproaches}

\subsection{Characteristic functions}

A subgroup $G\subseteq T^m$ acts freely on a moment-angle complex
$\Z_K$ if and only if $G$ intersects stabilizers of the action
$T^m\curvearrowright \Z_K$ trivially.

\begin{lemma}
Stabilizers of the action $T^m\curvearrowright \Z_K$ are
coordinate subtori $T^I\subseteq T^m$, corresponding to simplices
$I\in K$.
\end{lemma}

\begin{proof}
The subgroup $T^I$ preserves the point
$\{0\}^I\times\{1\}^{[m]\setminus I}\in(D^2,S^1)^I\subseteq \Z_K$.
\end{proof}

In this section we suppose for simplicity that $K$ does not have
ghost vertices. In other words, $\{i\}\in K$ for any $i\in[m]$.
Let $G\subset T^m$ be a toric subgroup of rank $s$ acting freely
on $\Z_K$. Consider the quotient map $\phi\colon T^m\to T^m/G$,
and fix an arbitrary coordinate representation $T^m/G\cong T^r$,
where $r=m-s$. We get a map $\phi\colon T^m\to T^r$ such that the
restriction $\phi|_{T^I}$ to any stabilizer subgroup is injective.
For each vertex $i\in[m]$ consider an $i$-th coordinate subgroup
$T^{\{i\}}\subset T^m$. Since $\{i\}\in K$, the subgroup
$\phi(T^{\{i\}})\subset T^r$ is $1$-dimensional, therefore
$\phi(T^{\{i\}}) =
(t^{\lambda_i^1},t^{\lambda_i^2},\ldots,t^{\lambda_i^r})$, where
$t\in T^1$ and $(\lambda_i^1,\lambda_i^2,\ldots,\lambda_i^r)\in
\Zo^r$. Define a map: $\Lambda\colon [m]\to\Zo^r$,
$\Lambda(i)=(\lambda_i^1,\lambda_i^2,\ldots,\lambda_i^r)$, called
\emph{characteristic map} (corresponding to the subgroup
$G\subseteq T^m$). Since $\phi|_{T^I}$ is injective the
characteristic map satisfies the condition:
\begin{equation}\label{eqStarCond}
\parbox[c]{10cm}{\center{If $I=\{i_1,\ldots,i_k\}\in K$, \\ then
$\Lambda(i_1),\ldots, \Lambda(i_k)$ form a part of a basis of the
lattice $\Zo^r$.}}\tag{$*$}
\end{equation}

Vice a versa any map $\Lambda\colon[m]\to \Zo^r$ satisfying
\eqref{eqStarCond} corresponds to some toric subgroup $G\subset
T^m$ of rank $s=m-r$ acting freely on $\Z_K$.

The case of real moment-angle complexes is similar. Each subgroup
$G\subset \Zt^m$ of rank $s$ acting freely on $\Zr_K$ determines a
map $\Lambda_{\Ro}\colon[m]\to \Zt^r$ with $r=m-s$. This map
satisfies the condition
\begin{equation}\label{eqStarCondR}
\parbox[c]{10cm}{\center{ If $I=\{i_1,\ldots,i_k\}\in K$, \\ then
$\Lambda(i_1),\ldots, \Lambda(i_k)$ are linearly independent in
$\Zt^r$.}}\tag{$*_\Ro$}
\end{equation}

These considerations prove the following statement.

\begin{stm}[I.Izmestiev \cite{Izm2}]\label{stmIzmestCrit}
Let $r(K)$ denote the minimal integer $r$ for which there exists a
map $[m]\to\Zo^r$ satisfying \eqref{eqStarCond}. Let $\rr(K)$
denote the minimal integer $r$ for which there exists a map
$[m]\to\Zt^r$ satisfying \eqref{eqStarCondR}. Then $s(K)=m-r(K)$
and $\sr(K)=m-\rr(K)$.
\end{stm}

\begin{rem} Note that actually there is no 1-to-1 correspondence
between freely acting subgroups and characteristic functions. The
first reason is a choice of an isomorphism $T^m/G\cong T^r$ which
was arbitrary. The second reason is that characteristic function
was defined only up to sign. Integral vectors
$(\lambda_i^1,\lambda_i^2,\ldots,\lambda_i^r)$ and
$-(\lambda_i^1,\lambda_i^2,\ldots,\lambda_i^r)$ determine the same
1-dimensional toric subgroup.
\end{rem}

\subsection{Generalized chromatic invariants}

Let $K$ and $L$ be simplicial complexes on sets $V(K)$ and $V(L)$,
possibly infinite. A map $f\colon V(K)\to V(L)$ is called a
simplicial map (or a map of simplicial complexes) if $I\in K$
implies $f(I)\in L$. For a simplicial map we write $f\colon K\to
L$. A map $f\colon K\to L$ is called non-degenerate if
$|f(I)|=|I|$ for each simplex $I\in K$. The following general
definition is due to R.\,\v{Z}ivaljevi\'{c} \cite[def. 4.11]{Ziv}.

\begin{defin}[Generalized chromatic
invariant]\label{definGenColoring} Let $\F=\{T_\alpha\mid\alpha\in
A\}$ be a family of ``test'' simplicial complexes and let
$\wt\colon A\to \Ro$ be a real-valued function. A
$T_\alpha$-coloring of $K$ is just a non-degenerate simplicial map
$f\colon K\to T_\alpha$ and $\gamma_{(\F,\wt)}$, the
$(\F,\wt)$-chromatic number of $K$, is defined as the infimum of
all weights over all $T_{\alpha}$-colorings,
\begin{equation}\label{eqGenChromNumber}
\gamma_{(\F,\wt)}(K) \eqd \inf\{\wt(\alpha)\mid \mbox{ there
exists a } T_\alpha\mbox{-coloring of } K\}
\end{equation}
If there are no colorings at all, set
$\gamma_{(\F,\wt)}(K)\eqd+\infty$.
\end{defin}

\begin{ex}\label{exDeltaFamily} Let
$\F_{\Delta}=\{\Delta_{[n]}\mid n\in\Noo\}$ be the family of
simplices weighted by numbers of vertices $\wt_{\Delta}(n)\eqd n$.
The $\F_{\Delta}$-coloring is a non-degenerate simplicial map
$f\colon K\to \Delta_{[n]}$. This is just a map $f\colon V(K)\to
[n]$ such that $f(i)\neq f(j)$ for $\{i,j\}\in K$. Thus, $f$ is a
coloring in classical sense and
$\gamma_{(\F_{\Delta},\wt_{\Delta})}(K) = \gamma(K)$ --- the
ordinary chromatic number.
\end{ex}

\begin{ex}\label{exDimenFamily} Consider the complex
$\Delta_{\infty}^{(n)}$ which has infinite countable set $\Omega$
of vertices and simplices --- all subsets $I\subset \Omega$ with
$|I|\leqslant n+1$. Consider the family
$\F_d=\{\Delta_{\infty}^{(n)}\mid n\in\Noo\}$ weighted by
$\wt_d(n)\eqd n$. Then, obviously,
$\gamma_{(\F_d,\wt_d)}(K)=\dim(K)$.
\end{ex}

\begin{ex} Many classical and new invariants in graph theory are
generalized chromatic invariants. These include fractional and
circular chromatic numbers \cite{Koz}, orthogonal colorings
\cite{Orth}, quantum chromatic number \cite{Quant}.
\end{ex}

\begin{ex}\label{exUnivCpxesFamily} An integral vector $v\in \Zo^n$,
$v\neq 0$ is called \emph{primitive} if $v$ is not divisible by
natural numbers other than $1$. A collection of integral vectors
$I=\{v_1,\ldots,v_k\}\subset\Zo^n$ is called \emph{unimodular} if
$I$ is a part of some basis of a lattice $\Zo^n$. Clearly, any
vector in a unimodular collection is primitive. A subcollection of
a unimodular collection is unimodular.

Consider the simplicial complex $U_n$ in which: (1) vertices are
primitive vectors of $\Zo^n$; (2) simplices are unimodular
collections of vectors. Obviously, maximal simplices are bases of
the lattice $\Zo^n$, so $\dim U_n=n-1$. Define the test family
$\F_{U}=\{U_n\mid n\in \Noo\}$ weighted by $\wt_U(n)\eqd n$. Then
an $(\F_{U},\wt_U)$-coloring of a complex $K$ is exactly the map
$\Lambda\colon[m]\to \Z^n$, which satisfies
\eqref{eqStarCond}-condition. Therefore, the generalized chromatic
invariant $\gamma_{(\F_{U},\wt_U)}(K)$ is exactly $r(K)=m-s(K)$.

Similarly, define $\Ur_n$ as a simplicial complex on the set
$\Zt^n\setminus\{0\}$ in which $I$ is a simplex if $I$ is a set of
binary vectors linearly independent over $\Zt$. Clearly, $\dim
\Ur_n = n-1$. Define the test family $\F_{U\Ro}=\{\Ur_n\mid n\in
\Noo\}$ weighted by $\wt_{U\Ro}(n)\eqd n$. Then
$\gamma_{(\F_{U\Ro},\wt_{U\Ro})}(K)=\rr(K)=m-\sr(K)$.
\end{ex}

We can always assume that test families satisfy
$\gamma_{(\F,\wt)}(T_\alpha)=\wt(T_\alpha)$ in Definition
\ref{definGenColoring}. This holds for the families described
above.

Generalized chromatic invariants share a common property. If there
exists a non-degenerate map $g\colon K\to L$, then
$\gamma_{(\F,\wt)}(K)\leqslant\gamma_{(\F,\wt)}(L)$. This fact
follows easily from the definition: if $f\colon L\to \F_\alpha$ is
an $\F_\alpha$-coloring of $L$, then $f\circ g\colon K\to L\to
\F_\alpha$ is an $\F_\alpha$-coloring of $K$ with the same weight.
For Buchstaber invariants (Example \ref{exUnivCpxesFamily}) this
observation gives
$$
s(K)\geqslant s(L)-m_L+m_K,\qquad \sr(K)\geqslant \sr(L)-m_L+m_K,
$$
where $m_K$, $m_L$ are the numbers of vertices of $K$ and $L$.
This fact was first pointed out by N.Erokhovets in \cite{Er}.

On the other hand, the aforementioned monotonicity property is in
general not substantial due of the following ``general nonsense''
argument.

\begin{claim}
Let $a(\cdot)$ be an invariant of simplicial complexes taking
values in $\Ro$ and such that $a(K)\leqslant a(L)$ if there exists
a non-degenerate map $g\colon K\to L$. Then $a(\cdot)$ is a
generalized chromatic invariant.
\end{claim}

\begin{proof}
Just take the family of all simplicial complexes weighted by
$a(\cdot)$ itself. Of course, we suppose that all complexes under
consideration belong to some ``good universe'' to avoid
set-theoretical problems.
\end{proof}

Let us describe the relation between different generalized
chromatic invariants. Let $(\F_1,\wt_1)$ and $(\F_2,\wt_2)$ be
weighted test families. We say that there is a morphism
$\Psi~\colon(\F_1,\wt_1) \to (\F_2,\wt_2)$ if for each complex
$T\in\F_1$ there exists a non-degenerate simplicial map from $T$
to some $S\in \F_2$ with $\wt_2(S)\leqslant\wt_1(T)$.

\begin{lemma}\label{lemmaMorphismOfFamilies}
If there is a morphism from $(\F_1,\wt_1)$ to $(\F_2,\wt_2)$ then
$\gamma_{(\F_2,\wt_2)}(K)\leqslant\gamma_{(\F_1,\wt_1)}(K)$ for
any $K$.
\end{lemma}

The proof is immediate.

\begin{lemma}\label{lemmaSeriesOfMorph}
There is a series of morphisms:
$$
(\F_\Delta,\wt_{\Delta})\to(\F_U,\wt_U)\to(\F_{U\Ro},\wt_{U\Ro})\to(\F_d,\wt_d+1)
$$
for the families defined in Examples
\ref{exDeltaFamily},\ref{exDimenFamily} and
\ref{exUnivCpxesFamily}
\end{lemma}
\begin{proof}
Indeed, for each $n\in \Noo$ we have the following. (1) A
non-degenerate map $\Delta_{[n]}~\to~U_n$, sending $[n]$ to a
basis of the lattice $\Zo^n$. (2) A non-degenerate map
$p~\colon~U_n~\to~\Ur_n$, which reduces each primitive vector (a
vertex of $U_n$) modulo $2$. The map $p$, obviously, sends
unimodular collections from $\Zo^n$ to linearly independent sets
in $\Zt^n$. (3) A non-degenerate inclusion map
$\Ur_n~\to~\Delta_{\infty}^{n-1}$.
\end{proof}

From Lemmas \ref{lemmaMorphismOfFamilies} and
\ref{lemmaSeriesOfMorph} follows
$$\dim K+1\leqslant \rr(K)\leqslant r(K)\leqslant \gamma(K),$$
for any $K$. Equivalently:
\begin{equation}\label{eqGenEstimateBuch}
m-\gamma(K)\leqslant s(K)\leqslant \sr(K)\leqslant m-\dim K-1.
\end{equation}

The estimation of $s(K)$ by ordinary chromatic number was first
proved in \cite{Izm1}. The inequality between real and ordinary
Buchstaber invariant can be understood topologically as
well~\cite{FM}.

We call two test families equivalent, $(\F_1,\wt_1)\sim
(\F_2,\wt_2)$, if there are morphisms in both directions:
$\Psi\colon (\F_1,\wt_1)\to (\F_2,\wt_2)$ and $\Phi\colon
(\F_2,\wt_2)\to (\F_1,\wt_1)$. Equivalent families define equal
generalized chromatic invariants by Lemma
\ref{lemmaMorphismOfFamilies}. Therefore to prove that certain
generalized chromatic invariants are different we need to prove
that their test families are not equivalent.

In particular, to prove that $r(\cdot)$ and $\rr(\cdot)$ are
different invariants, it is sufficient to show that for some $n\in
\Noo$ there is no non-degenerate map from $\Ur_n$ to $U_n$. In
other words, we should prove that $r(\Ur_n)>n=\rr(\Ur_n)$ for some
$n\in \Noo$. This consideration is summarized as follows:

\begin{claim}
If there exists a simplicial complex $K$ such that
$s(K)\neq\sr(K)$ then such complex can be found among $\{\Ur_n\mid
n\in \Noo\}$.
\end{claim}

We start to check complexes $\Ur_n$ for small values of $n$. For a
test family $(\F,\wt)$ define $(\F^{(\ell)},\wt)$ as a family of
$\ell$-skeletons of members of $\F$.

\begin{prop}\label{propDim012}
If $\dim K=0,1,2$, then $s(K)=\sr(K)$. In particular,
$s(\Ur_n)=\sr(\Ur_n)$ for $i=1,2,3$.
\end{prop}

\begin{proof}
For complexes $K$ of dimension $0,1,2$ a non-degenerate map from
$K$ to $T$ is the same as a non-degenerate map from $K$ to
$T^{(2)}$. Therefore, $\gamma_{(\F,\wt)}(K) =
\gamma_{(\F^{(2)},\wt)}(K)$.

We prove that $(\F_U^{(2)},\wt_U)\sim
(\F_{U\Ro}^{(2)},\wt_{U\Ro})$. The proof exploits a trick invented
in \cite{ErArx,ErNewBig}. The modulo 2 reduction map
$p~\colon~U_n^{(2)}~\to~\Ur_n^{(2)}$ is already constructed. Let
us construct a non-degenerate map
$q~\colon~\Ur_n^{(2)}~\to~U_n^{(2)}$. The vertex $v$ of $\Ur_n$ is
a vector in $\Zt^n\setminus\{0\}$. It can be written as an array
of $0$ and $1$. Consider $q(v)\in\Zo^n$
--- the same array of $0$ and $1$ as an integral vector. It is easily
shown that if $I\subset\Zt^n$ is a set of at most $3$ linearly
independent vectors, then $\{q(v)\mid v\in I\}$ is unimodular in
$\Zo^n$. Thus $q$ is a non-degenerate map from $\Ur_n^{(2)}$ to
$U_n^{(2)}$.

Finally, $r(K) = \gamma_{(\F_U,\wt_U)}(K) =
\gamma_{(\F_U^{(2)},\wt_U)}(K) =
\gamma_{(\F_{U\Ro}^{(2)},\wt_{U\Ro})}(K) =
\gamma_{(\F_{U\Ro},\wt_{U\Ro})}(K) = \rr(K)$. The proposition now
follows from Statement \ref{stmIzmestCrit}.
\end{proof}

More can be said in the case $\dim K\leqslant 1$.

\begin{prop}\label{propDim01}
If $\dim K \leqslant 1$ then $s(K)=\sr(K)=m-[\log_2\gamma(K)]-1$.
Here $m$ is the number of vertices of $K$, $\gamma(K)$ ---
chromatic number, and $[\,\cdot\,]$ denotes an integral part.
\end{prop}

\begin{proof}
Note that $\Ur_n^{(1)}\cong \Delta_{[2^n-1]}^{(1)}$, since both
complexes are complete graphs on $2^n~-~1$ vertices. Thus the
family $(\F_{U\Ro}^{(1)},\wt_{U\Ro})$ is equivalent to
$(\{\Delta_{[2^n-1]}^{(1)}\},\wt_{\Delta})$ which is the subfamily
of $(\F_{\Delta}^{(1)},\wt_{\Delta})$. Formula
$\rr(K)=[\log_2\gamma(K)]+1$ follows easily.
\end{proof}

\begin{cor}
For finite $1$-dimensional simplicial complexes (i.e. simple
graphs) the problem to decide, whether $s(K)$ (or $\sr(K)$) is
equal to $m-2$, is NP-complete.
\end{cor}

\begin{proof}
By Proposition~\ref{propDim01}, $\sr(K)=m-2$ if and only if
$\gamma(K)=2$ or $\gamma(K)=3$. $2$-colorability of a graph $K$
can be verified in polynomial time. $3$-colorability of a graph
$K$ is an NP-complete decision problem \cite[A1,GT4 in
Appendix]{GJ}.
\end{proof}

\section{Real and ordinary Buchstaber invariants are different}\label{SecRealComplex}

In this section we prove Theorem~\ref{thmRealComplex}, by showing
that $s(\Ur_4)>4=\sr(\Ur_4)$ for the complex $\Ur_4$ defined in
the previous section. In other words, we prove that there is no
non-degenerate simplicial map from $\Ur_4$ to $U_4$.

Let $e$ denote the nonzero element of $\Zt$ to avoid confusion
with integral unit. Recall the map $p\colon U_4\to \Ur_4$
described in Lemma \ref{lemmaSeriesOfMorph}. This map sends
$(x_1,x_2,x_3,x_4)\in\Zo^4$ to $(x_1, x_2, x_3,
x_4)\mod2\in\Zt^4$.

\begin{lemma}\label{lemNondegThusInj}
Let $f\colon \Ur_4 \to N$ be a non-degenerate map. Then $f$ is an
injective map of vertices.
\end{lemma}

\begin{proof}
Vertices of $\Ur_4$ are pairwise connected. By non-degeneracy,
$|\{f(v_1),f(v_2)\}|=2$, thus $f(v_1)\neq f(v_2)$.
\end{proof}

\begin{rem} Every non-degenerate map $f\colon \Ur_4 \to N$ is injective
on simplices as well.
\end{rem}

\begin{lemma}\label{lemGenPos}
If there exists a non-degenerate map $\tilde{\nu}\colon \Ur_4\to
U_4$, then there exists a non-degenerate map $\nu\colon \Ur_4\to
U_4$ such that
\begin{equation}\label{eqLiftDefin}
p\circ\nu = \id\colon \Ur_4\to \Ur_4.
\end{equation}
\end{lemma}

\begin{proof}
Consider a map $q = p\circ\tilde{\nu}\colon \Ur_4\to \Ur_4$. The
map $q$ is a non-degenerate simplicial map, therefore, by Lemma
\ref{lemNondegThusInj}, it is injective on vertices of $\Ur_4$.
Thus $q$ defines a permutation on a finite set of vertices
$V(\Ur_4)$. Then $q^n=\id$ for some $n\geqslant 1$. Take $\nu =
\tilde{\nu}\circ q^{n-1}\colon \Ur_4 \to U_4$. Then $\nu$ is a
non-degenerate simplicial map, and $p\circ \nu = q^n = \id$.
\end{proof}

A non-degenerate map $\nu\colon \Ur_4\to U_4$ will be called a
\emph{lift} if it satisfies \eqref{eqLiftDefin}. To prove the
theorem it is sufficient to prove that lifts do not exist.

Suppose the contrary. Let $\Lambda\colon \Ur_4\to U_4$ be a lift.
Vertices of $\Ur_4$ are, by definition, nonzero vectors of
$\Zt^4$. We list them in \eqref{eqGenSchemeU}. Vectors at the
right hand side of \eqref{eqGenSchemeU} are the values of
$\Lambda$. Each vector at the right is a primitive vector in
$\Zo^4$. Since $\Lambda$ is a lift, numbers $a_i$ are odd and
$b_i$ are even.

\begin{align}\label{eqGenSchemeU}
v_1=(e,0,0,0)&\longmapsto (a_1,b_1,b_2,b_3)\notag\\
v_2=(0,e,0,0)&\longmapsto (b_4,a_2,b_5,b_6)\notag\\
v_3=(0,0,e,0)&\longmapsto (b_7,b_8,a_3,b_9)\notag\\
v_4=(0,0,0,e)&\longmapsto (b_{10},b_{11},b_{12},a_4)\notag\\
v_5=(e,e,0,0)&\longmapsto (a_5,a_6,b_{13},b_{14})\notag\\
v_6=(e,0,e,0)&\longmapsto (a_7,b_{15},a_8,b_{16})\notag\\
v_7=(e,0,0,e)&\longmapsto (a_9,b_{17},b_{18},a_{10})\notag\\
v_8=(0,e,e,0)&\longmapsto (b_{19},a_{11},a_{12},b_{20})\\
v_9=(0,e,0,e)&\longmapsto (b_{21},a_{13},b_{22},a_{14})\notag\\
v_{10}=(0,0,e,e)&\longmapsto (b_{23},b_{24},a_{15},a_{16})\notag\\
v_{11}=(0,e,e,e)&\longmapsto (b_{25},a_{17},a_{18},a_{19})\notag\\
v_{12}=(e,0,e,e)&\longmapsto (a_{20},b_{26},a_{21},a_{22})\notag\\
v_{13}=(e,e,0,e)&\longmapsto (a_{23},a_{24},b_{27},a_{25})\notag\\
v_{14}=(e,e,e,0)&\longmapsto (a_{26},a_{27},a_{28},b_{28})\notag\\
v_{15}=(e,e,e,e)&\longmapsto (a_{29},a_{30},a_{31},a_{32})\notag
\end{align}

Values of $\Lambda$ should satisfy \eqref{eqStarCond}-condition.
It is reformulated for this particular case as follows:
\begin{equation}\label{eqStarCondS}
\parbox[c]{12cm}{\center{If $v_{i_1},v_{i_2},v_{i_3},v_{i_4}\in \Zt^4$
satisfy $\det(v_{i_1},v_{i_2},v_{i_3},v_{i_4})=e\in \Zt$,\\ then
$\det(\Lambda(v_{i_1}), \Lambda(v_{i_2}), \Lambda(v_{i_3}),
\Lambda(v_{i_4}))=\pm1\in\Zo$. }}\tag{$*$}
\end{equation}

\begin{lemma}\label{lemmaChangeSign}
Condition \eqref{eqStarCondS} is preserved under the change of
sign of any $\Lambda(v_i)$.
\end{lemma}

This is clear.

\begin{lemma}\label{lemmaWLOGbaseToBase}
Without loss of generality we may assume that
$\Lambda(v_1)=(1,0,0,0)$, $\Lambda(v_2)=(0,1,0,0)$,
$\Lambda(v_3)=(0,0,1,0)$, $\Lambda(v_4)=(0,0,0,1)$.
\end{lemma}

\begin{proof}
Indeed, $\det(v_1,v_2,v_3,v_4)=e$, therefore, by
\eqref{eqStarCondS}-condition, $\Lambda(v_i)_{i=1,2,3,4}$ is a
basis of the lattice $\Zo^4$. Expand all vectors $\Lambda(v_i)$ in
this basis.
\end{proof}

\begin{lemma}\label{lemmaPMUnits}
In the notation of \eqref{eqGenSchemeU} $a_i=\pm1$ for each
$i=1,\ldots,32$.
\end{lemma}

\begin{proof}
Consider the matrix $A$ over $\Zt$:
$$
A=\begin{pmatrix}
e&0&0&0\\
0&e&0&0\\
0&0&e&0\\
*&*&*&e
\end{pmatrix}\quad\rightsquigarrow\quad B=\begin{pmatrix}
1&0&0&0\\
0&1&0&0\\
0&0&1&0\\
*&*&*&a_i
\end{pmatrix}
$$
Since $\det(A)=e$, \eqref{eqStarCondS}-condition and Lemma
\ref{lemmaWLOGbaseToBase} imply $\det B =\pm1$. Therefore,
$a_i=\pm1$ if $a_i$ stands on the last position. Similar for other
positions of $a_i$.
\end{proof}

In particular, $\Lambda(v_{15}) = \Lambda((e,e,e,e)) =
(\pm1,\pm1,\pm1,\pm1)$.

\begin{lemma}
Without loss of generality we may assume that $\Lambda(v_{15}) =
(1,1,1,1)$.
\end{lemma}

\begin{proof}
Let $\Lambda(v_{15}) =
(\varepsilon_1,\varepsilon_2,\varepsilon_3,\varepsilon_4)$.
Consider a new basis of the lattice: $e_1'=\varepsilon_1e_1,
e_2'=\varepsilon_2e_2, e_3'=\varepsilon_3e_3,
e_4'=\varepsilon_4e_4$. Vector $\Lambda(v_{15})$ has coordinates
$(1,1,1,1)$ in the basis $\{e_1',e_2',e_3',e_4'\}$. Vectors
$\Lambda(v_1),\Lambda(v_2),\Lambda(v_3),\Lambda(v_4)$ have
coordinates $(\varepsilon_1,0,0,0)$, $(0,\varepsilon_2,0,0)$, etc.
We may change their signs, if necessary, by Lemma
\ref{lemmaChangeSign} and get $\Lambda(v_1)=(1,0,0,0)$,
$\Lambda(v_2)=(0,1,0,0)$, etc. as before.
\end{proof}

To summarize:

\begin{claim}
Without loss of generality, $\Lambda(v_1)=(1,0,0,0)$,
$\Lambda(v_2)=(0,1,0,0)$, $\Lambda(v_3)=(0,0,1,0)$,
$\Lambda(v_4)=(0,0,0,1)$, $\Lambda(v_{15})=(1,1,1,1)$ and
$a_i=\pm1$ for all $i=1,\ldots,32$ in \eqref{eqGenSchemeU}.
\end{claim}

Now we investigate which $b_i$ occur in (\ref{eqGenSchemeU}). A
new portion of notation is needed. From now on the bases of
$\Zt^4$ and $\Zo^4$ are fixed. For
$v=(\alpha_1,\alpha_2,\alpha_3,\alpha_4)\in\Zt^4$ the support is
defined as the set of positions with nonzero entries:
$\supp(v)\eqd\{i~\mid~\alpha_i=e\}$. Consider the standard Hamming
norm $\|v\|\eqd|\supp(v)|$. By definition, $\|v_i\|=2$ for
$5\leqslant i\leqslant 10$ and $\|v_i\|=3$ for $11\leqslant
i\leqslant 14$ in the notation of \eqref{eqGenSchemeU}.

Integral numbers standing in $\Lambda(v)$ at positions from
$\supp(v)$ are called \emph{odds} of $v$, numbers, standing at
other positions are called \emph{evens} of $v$. Thus, for example,
odds of $v_5$ are $\{a_5,a_6\}$ and its evens are
$\{b_{13},b_{14}\}$. Odds are odd numbers and evens are even
numbers as was mentioned before. Moreover, all odds are $\pm1$ by
Lemma \ref{lemmaPMUnits}. By Lemma \ref{lemmaChangeSign} we may
assume that the first odd of each $v_i\in\Zt^4\setminus\{0\}$ is
$1$. The vector $v_i\in \Zt^4\setminus\{0\}$ is called alternated
if $\Lambda(v_i)$ contains both $+1$ and $-1$ as odds. If $v_i$ is
not alternated, then all its odds are~$+1$.

\begin{lemma} \label{lemEvens02}
If $v_i\in\Zt^4$ is alternated, then all its evens are equal to
$0$. If $v_i\in\Zt^4$ is not alternated, then its evens are equal
to $0$ or $2$.
\end{lemma}

\begin{proof}
Consider the matrix:
$$
A=\begin{pmatrix}
e&0&0&0\\
0&e&0&0\\
e&e&e&e\\
*&*&0&e
\end{pmatrix}\quad\rightsquigarrow\quad B=\begin{pmatrix}
1&0&0&0\\
0&1&0&0\\
1&1&1&1\\
*&*&b_j&a_i
\end{pmatrix}
$$
Since $\det(A)=e$, \eqref{eqStarCondS}-condition implies $\det B
=\pm1$. Therefore, $a_i-b_j=\pm1$. If $a_i=1$, then $b_j$ is
either $0$ or $2$ which proves the second part of the statement.
If $a_i=-1$, then $b_j=0$ or $-2$. If $v$ is alternated, then each
$b_j$ should be either $0$ or $2$, and, on the other hand, it
should be either $0$ or $-2$ by the same reasons. Thus $b_j=0$ in
the alternated case.
\end{proof}

Lemma \ref{lemEvens02} reduces the task of finding characteristic
function from $\Ur_4$ to $\Zo^4$ to the finite enumeration of
possibilities. Each $a_i$ can be $1$ or $-1$ and $b_i$ can be $0$
or $2$. But still there are too many possibilities to use
computer-aided search; we want to simplify the task a bit more.

\begin{lemma}
Let $v_i,v_j\in\Zt^4$ be two different alternated vectors and
$\|v_i\|=\|v_j\|=2$. Then $\supp(v_i)\cap\supp(v_j)=\varnothing$.
\end{lemma}

\begin{proof}
Suppose that supports intersect. Without loss of generality $i=5$,
$j=8$. We have $\Lambda(v_5) = (1,-1,0,0)$ and $\Lambda(v_8) =
(0,1,-1,0)$ by Lemma \ref{lemEvens02}. Then
$$
\det\begin{pmatrix}
e&e&0&0\\
0&e&e&0\\
e&e&e&e\\
0&0&0&e
\end{pmatrix}=e\quad\rightsquigarrow\quad\det
\begin{pmatrix}
1&-1&0&0\\
0&1&-1&0\\
1&1&1&1\\
0&0&0&1
\end{pmatrix}=3,
$$
so the \eqref{eqStarCondS}-condition is violated. The
contradiction.
\end{proof}

Now we use the following algorithm to show that a lift $\Lambda$,
satisfying \eqref{eqStarCondS} does not exist. Consider three
possible cases depending on the number of alternated vectors among
$v_5,\ldots,v_{10}$: (a) there are no alternated vectors; (b)
there is exactly one alternated vector, say $v_5$; (c) there are
two alternated vectors with nonintersecting supports, say $v_5$
and $v_{10}$. In each case do the following \texttt{
\begin{enumerate}
\item Find all values of $(b_{13},\ldots,b_{24})\in\{0,2\}^{12}$ for
which \eqref{eqStarCondS}-condition is sa\-tis\-fied on the set
$\{v_1,\ldots,v_{10},v_{15}\}$.
\item For each output of the previous step check all
$(b_{25},\ldots,b_{28})\in\{0,2\}^4$ and
$(\widehat{a_{17}},a_{18},\ldots,\widehat{a_{20}},\ldots,\widehat{a_{23}},
\ldots, \widehat{a_{26}},\ldots,a_{28})\in\{1,-1\}^{8}$ for which
\eqref{eqStarCondS}-con\-di\-tion is satisfied on the whole set
$\{v_1,\ldots,v_{15}\}=\Zt^4\setminus\{0\}$.
\end{enumerate}
}

The task is split into two steps to reduce the time of
computation. GAP system \cite{GAP4} was used to perform this
calculation. The implementation of described algorithm shows that
there are no values of $a_i$ and $b_i$ for which
\eqref{eqStarCondS}-condition is satisfied. Thus $r(\Ur_4)>4$ and
$$
s(\Ur_4)<11=\sr(\Ur_4),
$$
which was to be proved.


\section{Buchstaber number is not determined by bigraded Betti
numbers}\label{SecBettiBuch}

\subsection{Technique of the proof}

This section contains the proof of Theorem \ref{thmBettiBuch}. To
construct simplicial complexes with desired properties the
following ingredients are used:
\begin{itemize}
\item The characterization of the Buchstaber invariant in terms of
minimal non-simplices, found by N.\,Erokhovets.
\item The Taylor resolution of a Stanley--Reisner module. We use this
resolution to construct different simplicial complexes with equal
bigraded Betti numbers.
\end{itemize}

\subsection{Erokhovets criterium}

A subset $I\subseteq[m]$ is called a \emph{minimal non-simplex}
(or a missing face) of $K$ if $I\notin K$, but $J\in K$ for any
$J\subsetneq I$. The set of all minimal non-simplices of $K$ is
denoted by $N(K)$.

\begin{stm}[N.\,Erokhovets \cite{ErNew,ErNewBig}] \label{stmErokhCrit}
The following conditions are equivalent:
\begin{enumerate}
\item $s(K)\geqslant 2$;
\item $\sr(K)\geqslant 2$;
\item \label{eqErokhCrit} there exist $J_1,J_2,J_3\in N(K)$ such that $J_1\cap
J_2\cap J_3=\varnothing$. Sets $J_i$ are allowed to coincide.
\end{enumerate}
\end{stm}

The next example will be used in the proof of
Theorem~\ref{thmBettiBuch}.

\begin{figure}[h]
\begin{center}
\includegraphics[scale=0.2]{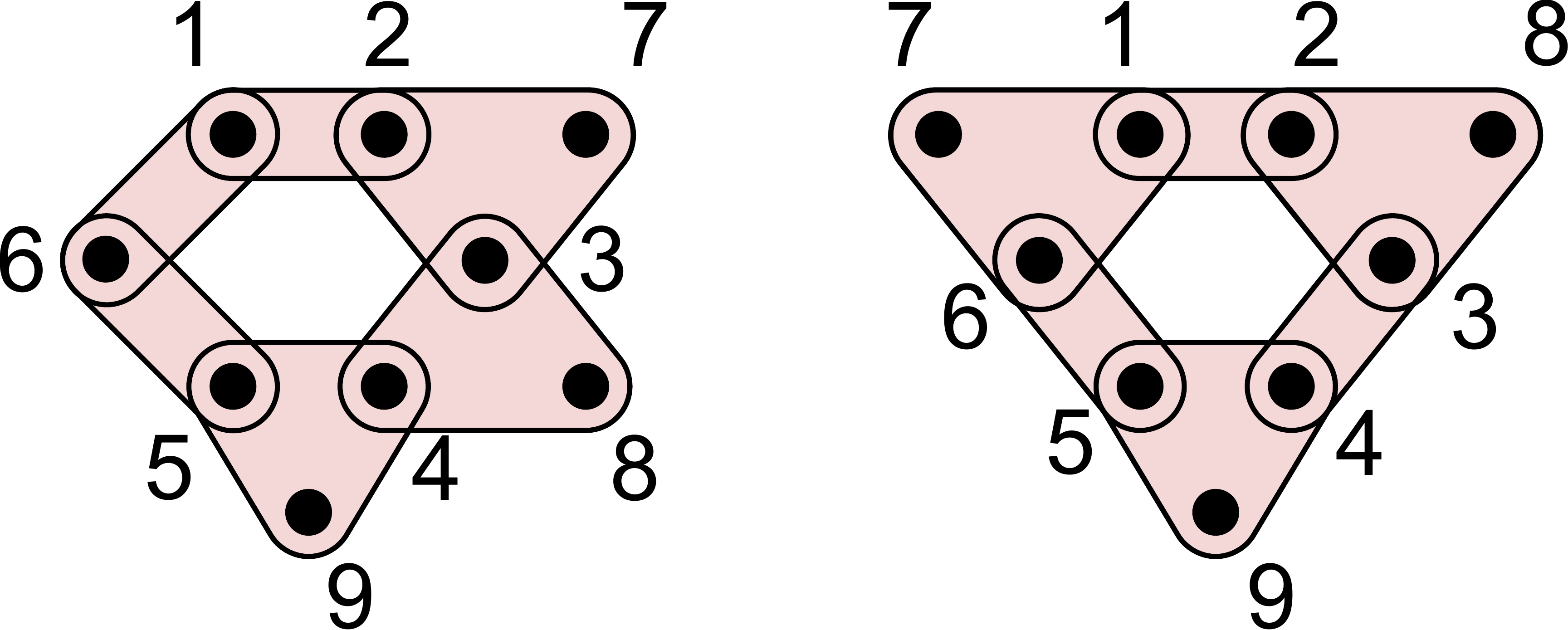}
\end{center}
\caption{Collections $\Cc_1$ and $\Cc_2$ of
subsets}\label{pictSetCollections}
\end{figure}

\begin{ex}\label{exSetColl} Let $\Sn\eqd\{1,2,\ldots,9\}$. Consider
two collections of subsets of $\Sn$ shown on
fig.\ref{pictSetCollections}. In the first collection there are no
$A_1,A_2,A_3\in \Cc_1$ such that $A_1\cup A_2\cup A_3=\Sn$. On the
other hand there exist $A_1,A_2,A_3\in \Cc_2$ such that $A_1\cup
A_2\cup A_3=\Sn$.

Now consider simplicial complexes $L_1$ and $L_2$ with
$N(L_i)=\{I\subset \Sn\mid \Sn\setminus I\in \Cc_i\}$ for $i=1,2$.
Then condition \eqref{eqErokhCrit} of Statement \ref{stmErokhCrit}
does not hold for $L_1$, but holds for $L_2$. Statement
\ref{stmErokhCrit} implies $s(L_1)=1$ and $s(L_2)\geqslant 2$ (and
the same for $\sr$).
\end{ex}

\begin{rem} One can consider collections $\Cc_1$ and $\Cc_2$ as
simplicial complexes. Then the complexes $L_i$ are Alexander duals
of $\Cc_i$ by the definition of combinatorial Alexander duality
(see e.g. \cite[Ex.2.26]{BP}).
\end{rem}

\subsection{Bigraded Betti numbers and Taylor resolution}

First, we review the basics of commutative algebra needed to
define bigraded Betti numbers.

Let $\ko$ be a ground field and $\ko[m]=\ko[v_1,\ldots,v_m]$ ---
the polynomial algebra graded by $\deg v_i=2$. Also define the
multigrading by $\mdeg(v_1^{n_1}\cdot\ldots\cdot
v_m^{n_m})=(2n_1,\ldots,2n_m)\in\Zo^m$. Denote by $\ko[m]^{+}$ the
maximal graded ideal of $\ko[m]$ --- i.e. the ideal generated by
monomials of positive degrees.

The \emph{Stanley--Reisner algebra} (otherwise called the
\emph{face ring}) of a simplicial complex $K$ on $m$ vertices is
the quotient algebra $\ko[K]=\ko[m]/I_{SR}(K)$, where $I_{SR}(K)$
is the square-free ideal generated by monomials, corresponding to
non-simplices of $K$:
$$
I_{SR}(K)=(v_{i_1}\cdot\ldots\cdot v_{i_k}\mid
\{i_1,\ldots,i_k\}\notin K);
$$
Both $\ko$ and $\ko[K]$ carry the structure of (multi)graded
$\ko[m]$-modules via quotient epimorphisms
$\ko[m]\to\ko[m]/\ko[m]^{+}\cong\ko$ and $\ko[m]\to\ko[K]$. Then
$\Tor^{*,*}_{\ko[m]}(\ko[K],\ko)$ is a $\Tor$-functor of
(multi)graded modules $\ko[K]$ and $\ko$. Recall its standard
construction in homological algebra.

\begin{con}\label{conTorViaResol} To describe
$\Tor^{*,*}_{\ko[m]}(\ko[K],\ko)$ do the following:
\begin{enumerate}
\item Take any free resolution of the module $\ko[K]$ by
(multi)graded $\ko[m]$-modules:
$$
\ldots \xrightarrow{d} R^{-\ell} \xrightarrow{d}
R^{-\ell+1}\xrightarrow{d}\ldots\to R^{-1}\to R^{0}\xrightarrow{d}
\ko[K]\qquad (\R)
$$
\item apply the functor $\otimes_{\ko[m]}\ko$;
\item take cohomology of the resulting complex:
$$
\Tor^{*,*}_{\ko[m]}(\ko[K],\ko)\congd
H^*(\R^*\otimes_{\ko[m]}\ko;d\otimes_{\ko[m]}\ko)
$$
\end{enumerate}
The resulting vector space inherits inner (multi)grading from
$\R$. It also obtains an additional grading $-\ell$. It is well
known that $\Tor^{*,*}_{\ko[m]}(\ko[K],\ko)\cong
\bigoplus_{(\ell,\bar{j})\in\Zo^{m+1}}\Tor^{-\ell,2\bar{j}}_{\ko[m]}(\ko[K],\ko)$
does not depend on the choice of a free (multi)graded resolution
$\R$. Define \emph{bigraded Betti numbers} of $K$ as
$$\beta^{-\ell,2j}(K)\eqd\dim_{\ko}\Tor_{\ko[m]}^{-\ell,2j}(\ko[K],\ko).$$
\end{con}

\begin{defin}[Minimal resolution]
A resolution $\R$ is called \emph{minimal} if $\im(d)\subset
\ko[m]^{+}\cdot\R$, or, equivalently, $d\otimes_{\ko[m]}\ko = 0$.
\end{defin}

For a minimal resolution $\R$ step (3) in Construction
\ref{conTorViaResol} is skipped. Therefore:
$$\beta^{-\ell,2j}(K)= \mbox{number of generators of the module }R^{-\ell}\mbox{ in degree } 2j.$$

Several explicit constructions of free resolutions of $\ko[K]$ are
known. In our considerations we use one of the most important and
basic constructions --- the Taylor resolution. In general, Taylor
resolution is defined for any monomial ideal (see \cite{MS} or
\cite{Merm}). Here we concentrate only on the case of
Stanley--Reisner rings, i.e. the case of square-free monomial
ideals. The work \cite{Wang} is also concerned with this
particular case.

In the sequel the following convention is used. Any subset
$B\subseteq [m]$ determines the vector $\delta_B\in \Zo^m$ with
$i$-th coordinate equal to $1$ if $i\in B$ and $0$ otherwise. We
simply write $B\in \Zo^m$ meaning $\delta_B\in \Zo^m$. For a set
$B$ we denote the monomial
$\overline{v}^{\delta_B}=v_1^{\delta_B^1}\ldots
v_m^{\delta_B^m}\in \ko[m]$ simply by $v^B$.

\begin{con}[Taylor resolution] \label{conTaylorRes} Consider the set
$N(K)$ of minimal non-simplices. Fix a linear order on $N(K)$. For
each $J\in N(K)$ associate a formal variable $w_J$ and construct
the free $\ko[m]$-module
$$R_T^{-\ell}\eqd\Lambda^{\ell}[\{w_J\}]\otimes\ko[m]$$
Here $\Lambda^{\ell}[\{w_J\}]$ is the vector space over $\ko$,
generated by formal expressions
$W_\sigma=w_{J_1}\wedge\ldots\wedge w_{J_\ell}$ for all subsets
$\sigma=\{J_1,\ldots, J_\ell\}\subseteq N(K)$,
$J_1<\ldots<J_\ell$.

Define the multigrading
\begin{equation}\label{eqMdegOfW}
\mdeg(w_{J_1}\wedge\ldots\wedge w_{J_\ell}) \eqd \left(-\ell,
2\bigcup_{i=1}^{\ell}J_i\right)\in \Zo\times\Zo^{m},
\end{equation}
and the double grading
$$
\bideg(w_{J_1}\wedge\ldots\wedge w_{J_\ell}) \eqd \left(-\ell,
2\left|\bigcup_{i=1}^{\ell}J_i\right|\right) \in \Zo^2.
$$
The first component is called a homological grading.

Define the differential of $\ko[m]$-modules $d_T\colon
R_T^{-\ell}\to R_T^{-\ell+1}$ on the generators
$W_{\sigma}=w_{J_1}\wedge\ldots\wedge w_{J_\ell}$ by
\begin{equation}\label{eqTaylorDif}
d_T(w_{J_1}\wedge\ldots\wedge w_{J_\ell}) \eqd
\sum\limits_{i=1}^{\ell}(-1)^{i+1} v^{X_{\sigma,J_i}}\cdot
w_{J_1}\wedge\ldots \widehat{w_{J_i}} \ldots\wedge w_{J_\ell},
\end{equation}
where $v^{X_{\sigma,J_i}}\in \ko[m]$ is the monomial corresponding
to the set
$$X_{\sigma, J_i}\eqd J_i\setminus \left(J_1\cup\ldots
\widehat{J_i}\ldots\cup J_\ell\right) \subset [m].$$

Define the multiplication on the $\ko[m]$-module
$\R_T=\bigoplus_{\ell} R_T^{-\ell}$ by describing the products of
generators. Let $\sigma=\{J_1<\ldots<J_\ell\},
\tau=\{I_1<\ldots<I_k\}\subseteq N(K)$.
\begin{equation}\label{eqTaylorMultGener}
W_{\sigma}\cdot W_{\tau}\eqd \begin{cases}0, \mbox{ if }
\sigma\cap\tau\neq\varnothing;\\
\sgn(\sigma,\tau)v^{Y_{\sigma,\tau}} W_{\sigma\sqcup\tau}, \mbox{
otherwise.}
\end{cases}
\end{equation}
Here $v^{Y_{\sigma,\tau}}\in \ko[m]$ is the monomial corresponding
to the set of indices $Y_{\sigma,\tau} =
(\bigcup_{\sigma}J_i)\cap(\bigcup_{\tau}I_i)$. The sign
$\sgn(\sigma,\tau)$ is the sign of the permutation needed to sort
the unordered set $(J_1,\ldots,J_\ell,I_1,\ldots,I_k)$.
\end{con}

\begin{prop}[\cite{MS},\cite{Merm}]\mbox{}

\begin{enumerate}
\item The vector space $\R_T=\bigoplus_{\ell} R_T^{-\ell}$ is a
differential $\Zo^{m+1}$-graded algebra over the ring $\ko[m]$
w.r.t. to the multigrading, the differential, and the
multiplication described above. This algebra is skew-commutative
with respect to homological grading.

\item $H^{-\ell}(\R_T,d)=0$ if $\ell>0$. $H^0(\R_T,d)\cong\ko[K]$
as $\ko[m]$-algebras.
\end{enumerate}
Therefore, $\R_T$ is a free multiplicative resolution of a
Stanley--Reisner algebra $\ko[K]$.
\end{prop}

\begin{ex}\label{exKoszul} Let $o_m$ be a simplicial complex on a set
$[m]$ in which all vertices are ghost. We have $\ko[o_m]\cong\ko$
and $N(o_m)=[m]$. The Taylor resolution in this case is given by
$R_T^{-\ell}=\Lambda^{\ell}[u_1,\ldots,u_m]\otimes\ko[m]$, where
formal variables $u_i$ correspond to elements of $N(o_m)=[m]$ and
$\bideg u_i=(-1,2)$. By looking at general definitions of
differential and product we see that $\R_T$ is isomorphic to
$\Lambda[u_1,\ldots,u_m]\otimes\ko[m]$ with the standard Grassmann
product, and the differential given by $du_i=v_i$. In this example
we get the multiplicative resolution
$\Lambda[u_1,\ldots,u_m]\otimes\ko[m]$ of the $\ko[m]$-module
$\ko$. This resolution is widely known as \emph{Koszul
resolution}.
\end{ex}

\begin{ex} Let $K$ be the boundary of a square. Its maximal simplices
are $\{1,2\}$, $\{2,3\}$, $\{3,4\}$, $\{1,4\}$. In this case
$N(K)=\{\{1,3\}, \{2,4\}\}$. The Taylor resolution has the form
$$
\xymatrix{ \Lambda^{(2)}[w_{\{1,3\}},w_{\{2,4\}}]\otimes\ko[4]
\ar@{->}[r]^{d_2} \ar@{=}[d]&
\Lambda^{(1)}[w_{\{1,3\}},w_{\{2,4\}}]\otimes\ko[4]
\ar@{->}[r]^(0.7){d_1} \ar@{=}[d]& \ko[4]\cdot
1\ar@{->>}[r]&\ko[K]
\\
W_{\{\{1,3\},\{2,4\}\}}\cdot\ko[4]& w_{\{1,3\}}\cdot\ko[4]\oplus
w_{\{2,4\}}\cdot\ko[4]&& }
$$
with the multigrading
\begin{align*}
\mdeg(w_{\{1,3\}})&=(-1;(2,0,2,0)),\\
\mdeg(w_{\{2,4\}})&=(-1;(0,2,0,2)),\\
\mdeg(W_{\{\{1,3\},\{2,4\}\}})&=(-2;(2,2,2,2)),
\end{align*}
the differentials
\begin{align*}
d_1(w_{\{1,3\}})&=v_1v_3\cdot 1,\\
d_1(w_{\{2,4\}})&=v_2v_4\cdot 1, \\
d_2(W_{\{\{1,3\},\{2,4\}\}})&= v_1v_3\cdot w_{\{2,4\}} -
v_2v_4\cdot w_{\{1,3\}},
\end{align*}
and the product $w_{\{1,3\}}\times w_{\{2,4\}} =
W_{\{\{1,3\},\{2,4\}\}} = -w_{\{2,4\}}\times w_{\{1,3\}}$.
Clearly, $\im(d_2)=\ker(d_1)$ and $\im(d_1)=I_{SR}(K)$.
\end{ex}

\begin{ex}\label{exProdSimpl} Let $\Delta_{M}$ denote a simplex on a
set $M\neq\varnothing$. Consider
$K=\partial~\Delta_{M_1}~\ast~\ldots~\ast~\partial~\Delta_{M_k}$
--- a simplicial sphere on a set $M_1\sqcup\ldots\sqcup M_n$. Then
$N(K)=\{M_1,\ldots,M_n\}$. The Taylor resolution of $K$ is a
differential algebra
$$
\Lambda^*[w_1,\ldots,w_n]\otimes\ko[M_1\sqcup\ldots\sqcup M_n]
$$
with the standard Grassmann product, $\bideg(w_i)=(-1,2|M_i|)$,
and differential:
$$
d_T(w_{i_1}\wedge\ldots\wedge
w_{i_\ell})=\sum\limits_{k=1}^\ell(-1)^{k+1}v^{M_{i_k}}
w_{i_1}\wedge\ldots \widehat{w_{i_k}}\ldots \wedge w_{i_\ell}.
$$
The Taylor resolution is minimal, therefore
$\Tor^{*,*}_{\ko[M_1\sqcup\ldots\sqcup M_n]}(\ko[K];\ko)\cong
\Lambda^*[w_1,\ldots,w_n]$. Both previous examples are the
particular cases of this one.
\end{ex}

\subsection{Multiplication in $\Tor$}

\begin{con}\label{conMultInTor} There is a standard way to understand
the structure of $\Tor^{*,*}_{\ko[m]}(\ko[K];\ko)$ using Koszul
resolution. At first, note that
$\Tor^{*,*}_{\ko[m]}(\ko[K];\ko)\cong\Tor^{*,*}_{\ko[m]}(\ko;\ko[K])$.
By construction, $$\Tor^{*,*}_{\ko[m]}(\ko;\ko[K])\cong
H^*(\R\otimes_{\ko[m]}\ko[K]; d\otimes_{\ko[m]}\ko[K]),$$ where
$(\R^*,d)$ is any graded free resolution of $\ko$ as a
$\ko[m]$-module. Take for example Koszul resolution
$\R^{-\ell}\cong \Lambda[u_1,\ldots,u_m]\otimes\ko[m]$ with
grading and differential as described in Example \ref{exKoszul}.
Then
\begin{equation}
\Tor^{*,*}_{\ko[m]}(\ko;\ko[K])\cong
H^*(\Lambda[u_1,\ldots,u_m]\otimes\ko[K];
d\otimes_{\ko[m]}\ko[K]).
\end{equation}
The differential complex $\Lambda[u_1,\ldots,u_m]\otimes\ko[K]$
has the structure of a differential graded algebra. Thus
$\Tor^{*,*}_{\ko[m]}(\ko;\ko[K])$ has the structure of an algebra
as well. The word ``$\Tor$-algebra'' usually refers to this
definition of a multiplication.
\end{con}

\begin{stm}[\cite{BP2,Franz}]
The cohomology ring $H^*(\Z_K;\ko)$ is isomorphic as a graded
algebra to the $\Tor$-algebra $\Tor^{*,*}_{\ko[m]}(\ko[K];\ko)$
with the total grading $(-i,2j)\rightsquigarrow 2j-i$.
\end{stm}

\begin{rem}\label{remMultSame} According to Construction
\ref{conTorViaResol},
\begin{equation}\label{eqTorTaylor}
\Tor^{*,*}_{\ko[m]}(\ko[K];\ko)\cong
H^*(\R_T\otimes_{\ko[m]}\ko;d_T\otimes_{\ko[m]}\ko),
\end{equation}
where $(\R_T,d_T)$ is the Taylor resolution of $\ko[K]$. The
differential complex $\R_T\otimes_{\ko[m]}\ko$ obtains the
multiplication from the multiplication in the Taylor resolution.
This, in turn, induces the multiplication on
$H^*(\R_T\otimes_{\ko[m]}\ko;d_T\otimes_{\ko[m]}\ko)$. A priori it
is not clear, whether this multiplication on
$\Tor^{*,*}_{\ko[m]}(\ko[K];\ko)$ is the same as given by
Construction \ref{conMultInTor} or not. Fortunately, this
multiplicative structures are indeed the same (see e.g.
\cite[Constr. 2.3.2]{Avr}). So far the cohomological product in
$H^*(\Z_K;\ko)$ in some cases can be described in terms of the
Taylor resolution \cite{Wang}.
\end{rem}

\subsection{Taylor resolutions and minimality}

When the Taylor resolution is minimal, the benefits of both
notions --- Taylor resolution and minimality --- can be used.

\begin{lemma}\label{lemmaTayMin}
Let $K$ be a simplicial complex on $[m]$ and $N(K)$ --- the set of
minimal non-simplices. The following conditions are equivalent:
\begin{enumerate}
\item The Taylor resolution $(\R_T,d_T)$ of $\ko[K]$ is minimal.
\item Any minimal simplex $J\in N(K)$ is not a subset in the union of
others:
\begin{equation}\label{eqMinimalityCond}
J\nsubseteq \bigcup_{I\in N(K), I\neq J}I.
\end{equation}
\end{enumerate}
\end{lemma}

\begin{proof}
By definition, $\R_T$ is minimal if $d_T(R_T^{-\ell})\subseteq
\ko[m]^{+}\cdot R_T^{-\ell+1}$ for each $\ell\geqslant 0$. From
\eqref{eqTaylorDif} follows that $d_T(R_T^{-\ell})\subseteq
\ko[m]^{+}\cdot R_T^{-\ell+1}$ if and only if $v^{X_{\sigma,J}}\in
\ko[m]^{+}$ for each $\sigma\subseteq N(K)$ and $J\in\sigma$. This
is equivalent to $X_{\sigma,J}\neq\varnothing$. By definition,
$X_{\sigma,J}=J\setminus \left(\bigcup_{I\in \sigma, I\neq J}
I\right)$. If the Taylor resolution is minimal, then, in
particular, $X_{N(K),J}\neq\varnothing$, which is precisely the
condition \eqref{eqMinimalityCond} of the lemma. On the other
hand, $X_{N(K),J}\neq\varnothing$ implies
$X_{\sigma,J}\neq\varnothing$ for any $\sigma\subseteq N(K)$.
\end{proof}

\begin{lemma}\label{lemMinTayTorStruc}
If the Taylor resolution of $\ko[K]$ is minimal, then
$\Tor^{*,*}_{\ko[m]}(\ko[K],\ko)$ has the following description:
\begin{itemize}
\item It is generated as a vector space over $\ko$ by $W_{\sigma}$ for
$\sigma\subseteq N(K)$;
\item The multidegree is given by \eqref{eqMdegOfW};
\item The multiplication is given by
\begin{equation}
W_{\sigma}\cdot W_{\tau}=\begin{cases}
\sgn(\sigma,\tau)W_{\sigma\sqcup\tau}, \mbox{ if }
\sigma\cap\tau=\varnothing \mbox{ and }
(\bigcup_{J\in\sigma}J)\cap(\bigcup_{I\in\tau}I)=\varnothing \\
0, \mbox{ otherwise.}
\end{cases}
\end{equation}
\end{itemize}
\end{lemma}

The proof follows easily from the construction of Taylor
resolution and the definition of minimality.

For complexes with the minimal Taylor resolution bigraded Betti
numbers are expressed in combinatorial terms.

\begin{equation}\label{eqBetaForMiniTayl}
\beta^{-\ell,2j}(K)= \#\left\{\sigma\subseteq N(K) \middle|
|\sigma|=\ell, \left|\bigcup_{J\in\sigma}J\right|=j\right\}
\end{equation}

\subsection{Proof of Theorem \ref{thmBettiBuch}} At last, we have
all necessary ingredients to prove Theorem~\ref{thmBettiBuch}. As
a starting point take complexes $L_1$ and $L_2$ defined in Example
\ref{exSetColl}. Our plan is the following:

\begin{enumerate}
\item To upgrade $L_1$ and $L_2$ to the new complexes $K_1$ and $K_2$
satisfying condition \eqref{eqMinimalityCond} (Taylor resolution
is minimal);
\item To prove that
$\beta^{-\ell,2j}(K_1)=\beta^{-\ell,2j}(K_2)$ using formula
\eqref{eqBetaForMiniTayl};
\item To prove that $s(K_1)=1$ and $s(K_2)\geqslant 2$.
\item Final technical remarks: $\dim(K_1)=\dim(K_2)$,
$\gamma(K_1)=\gamma(K_2)$, and algebra isomorphism
$\Tor_{\ko[m]}(\ko[K_1],\ko) \cong \Tor_{\ko[m]}(\ko[K_2],\ko)$.
\end{enumerate}

\textbf{Step 1.} Let $L$ be any complex on a set $[m]$ with the
set of minimal non-simplices $N(L)$. For each $J\in N(L)$ consider
a symbol $a_J$. Define the complex $\res{L}$ on the set
$V=[m]\sqcup\{a_J\mid J\in N(K)\}$ with the set of minimal
non-simplices given by
\begin{equation}\label{eqLresolved}
\res{J}\in N(\res{L}) \Leftrightarrow
\res{J}=J\sqcup\{a_J\}\subset V \mbox{ for } J\in N(K)
\end{equation}
The Taylor resolution of the complex $\res{L}$ is minimal. Indeed,
any $\res{J}\in N(\res{L})$ contains the vertex $a_J$ which does
not belong to other minimal non-simplices. Therefore, condition
\eqref{eqMinimalityCond} holds for $\res{L}$.

Now we apply this construction to $L_1$ and $L_2$. Recall that
$N(L_i)=\{I\subset \Sn\mid \Sn\setminus I\in \Cc_i\}$ for $i=1,2$
and collections of subsets $\Cc_1,\Cc_2$ shown on
fig.\ref{pictSetCollections}. Set $K_i = \res{L_i}$ for $i=1,2$.
Both $K_1$ and $K_2$ have $9+6=15$ vertices.

\textbf{Step 2.} Apply \eqref{eqBetaForMiniTayl} to $K_i$:
\begin{multline}\label{eqBetaResolved}
\beta^{-\ell,2j}(K_i)= \#\left\{\sigma\subseteq N(K_i) \middle|
|\sigma|=\ell, \left|\bigcup_{\res{J}\in\sigma}\res{J}\right|=j\right\}=\\
\#\left\{\sigma\subseteq N(L_i) \middle| |\sigma|=\ell,
\left|\bigcup_{J\in\sigma}\res{J}\right|=j\right\}.
\end{multline}
The last equality is the consequence of bijective correspondence
between $N(L_i)$ and $N(K_i)$, sending $J\in N(L_i)$ to
$\res{J}\in N(K_i)$. We have
$$
\bigcup_{J\in\sigma}\res{J}=\bigcup_{J\in\sigma}(J\sqcup\{a_J\}) =
\left(\bigcup_{J\in\sigma}J\right)\sqcup\{a_J\mid J\in\sigma\},
$$
therefore
$$
\left|\bigcup_{J\in\sigma}\res{J}\right| =
\left|\bigcup_{J\in\sigma}J\right| + |\sigma|.
$$
Returning to \eqref{eqBetaResolved},
\begin{multline}
\beta^{-\ell,2j}(K_i)=\#\left\{\sigma\subseteq N(L_i)\middle|
|\sigma|=\ell, \left|\bigcup_{J\in\sigma}\res{J}\right|=j\right\} = \\
= \#\left\{\sigma\subseteq N(L_i)\middle| |\sigma|=\ell,
\left|\bigcup_{J\in\sigma}J\right|=j-\ell\right\}=\\
= \#\left\{\sigma\subseteq \Cc_i\middle| |\sigma|=\ell,
\left|\bigcap_{A\in\sigma}A\right|=9-(j-\ell)\right\}.
\end{multline}
The last equality follows from the definition of $L_i$, since
$N(L_i)$ consists of complements to subsets of the collection
$\Cc_i$. By analyzing fig.\ref{pictSetCollections} we see that for
each $\ell$ and $j$
$$
\#\left\{\sigma\subseteq \Cc_1 \middle| |\sigma|=\ell,
\left|\bigcap_{A\in\sigma}A\right|=9-(j-\ell)\right\}=
\#\left\{\sigma\subseteq \Cc_2 \middle| |\sigma|=\ell,
\left|\bigcap_{A\in\sigma}A\right|=9-(j-\ell)\right\}.
$$
Indeed, in both $\Cc_1$ and $\Cc_2$ there are $3$ subsets of
cardinality $2$, $3$ subsets of cardinality $3$, $6$ pairwise
intersections of cardinality $1$, and all other intersections are
empty. Therefore, $\beta^{-\ell,2j}(K_1)=\beta^{-\ell,2j}(K_2)$.
The nonzero bigraded Betti numbers calculated by the described
method are presented in fig.\ref{pictBettiSetCol} (empty cells are
filled with zeroes).

\begin{figure}[h]
\begin{center}
\includegraphics[scale=0.4]{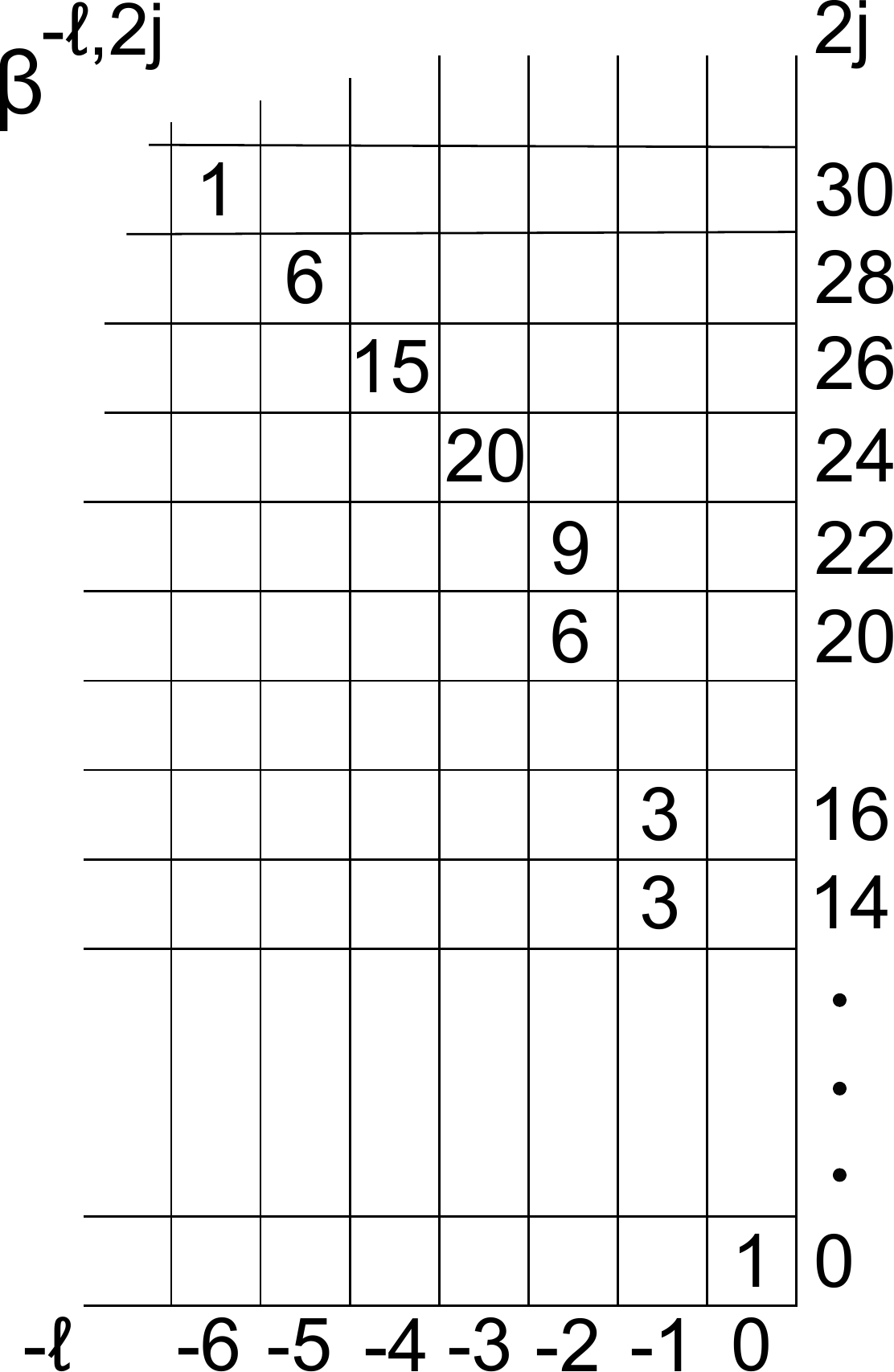}
\end{center}
\caption{Bigraded Betti numbers of $K_1$ and
$K_2$}\label{pictBettiSetCol}
\end{figure}

\textbf{Step 3.} We use the following simple observation.
Condition \eqref{eqErokhCrit} of Statement \ref{stmErokhCrit}
holds for the complex $L$ whenever it holds for $\res{L}$. Indeed,
$\res{J_1}\cap\res{J_2}\cap\res{J_3} = (J_1\sqcup\{a_{J_1}\})\cap(
J_2\sqcup\{a_{J_2}\})\cap (J_3\sqcup\{a_{J_1}\})=J_1\cap J_2\cap
J_3$. As observed in Example \ref{exSetColl} condition
\eqref{eqErokhCrit} of Statement \ref{stmErokhCrit} holds for
$L_2$ and does not hold for $L_1$. Therefore it also holds for
$K_2=\res{L_2}$ and does not hold for $K_1=\res{L_1}$. Thus
$s(K_1)\neq s(K_2)$ and $\sr(K_1)\neq \sr(K_2)$.

\textbf{Step 4.} Final remarks.

\begin{rem} Let us prove that $\dim K_1=\dim K_2=12$. Consider the
complement to the set $\{1,4\}$ in the set of vertices of $K_1$
(see fig. \ref{pictSetCollections}):
\begin{equation*}
S=\{1,2,\ldots,9,a_1,\ldots,a_6\}\setminus \{1,4\}.
\end{equation*}
Suppose that $S\notin K_1$. Then there exists $\res{J}\in N(K_1)$
such that $\res{J}\subseteq S$. Therefore, $\{1,4\}=\Sn\setminus
S\subseteq \Sn\setminus \res{J}$. By construction, $\Sn\setminus
\res{J}\in \Cc_1$. But $\{1,4\}$ is not contained in any $A\in
\Cc_1$ --- the contradiction. Thus $S\in K_1$ and $\dim
K_1\geqslant |S|-1=12$. Similar reasoning shows that there is no
simplex in $K_1$ of cardinality $14$ (because any singleton lies
in some $A\in \Cc_1$). Therefore $\dim K_1$ is exactly $12$. Same
for $K_2$.
\end{rem}

\begin{rem} In both complexes $K_1$ and $K_2$ there are no minimal
non-simplices of cardinality $1$ and $2$. Therefore all pairs of
vertices in $K_1$ and $K_2$ are connected by edges, so
$1$-skeletons $K_1^{(1)}$, $K_2^{(1)}$ are complete graphs on $15$
vertices. Thus chromatic numbers coincide
$\gamma(K_1)=\gamma(K_2)=15$.
\end{rem}

\begin{rem} $\Tor$-algebras of $K_1$ and $K_2$ are isomorphic as
algebras. Actually, the products in $\Tor_{\ko[15]}(\ko[K_1],\ko)$
and $\Tor_{\ko[15]}(\ko[K_2],\ko)$ are trivial by dimensional
reasons (see fig. \ref{pictBettiSetCol}): products of nonzero
elements always hit zero cells. The triviality of multiplication
can be deduced also from Lemma \ref{lemMinTayTorStruc} but this
approach requires a complicated combinatorial reasoning.
\end{rem}

These remarks conclude the proof of Theorem~\ref{thmBettiBuch}.

\subsection{Other invariants coming from $\Z_K$}
\begin{rem}
Question \ref{queInvProblem} is answered in the negative if
$A(\cdot)$ is a collection of bigraded Betti numbers. We may ask
the same question for $A(\cdot)$ --- the collection of all
multigraded Betti numbers $\beta^{-i,2\overline{j}}(K)\eqd\dim
\Tor_{\ko[m]}^{-i,2\overline{j}}(\ko[K],\ko)$.

Eventually, this question does not make sense. Multigraded Betti
numbers are too strong invariants:
$\beta^{-1,2\overline{j}}(K)=\beta^{-1,2\overline{j}}(L)$ implies
$K=L$. Indeed, for a subset $A\subseteq[m]$ the condition
$\beta^{-1,2A}(K)\neq 0$ is equivalent to $A\in N(K)$ by the
construction of the Taylor resolution (also by Hochster's formula
\cite[Th.3.2.9]{BPnew}). Therefore multigraded Betti numbers
encode all minimal non-simplices and determine the complex $K$
uniquely.
\end{rem}

\begin{rem}
Question \ref{queInvProblem} may be formulated for an equivariant
cohomology ring of $\Z_K$. This task is not interesting as well.
Indeed, $H^*_{T^m}(\Z_K;\ko)\cong \ko[K]$ (see \cite{DJ} or
\cite{BP2}). It is known, that the Stanley--Reisner algebra
$\ko[K]$ determines the combinatorics of $K$ uniquely~\cite{BG}.
Therefore multiplicative isomorphism $H^*_{T^m}(\Z_{K_1};\ko)\cong
H^*_{T^m}(\Z_{K_2};\ko)$ implies $K_1\cong K_2$ and, in
particular, $s(K_1)=s(K_2)$.
\end{rem}

\section{Conclusion and open questions}\label{SecConclusion}

Constructions of Buchstaber invariants and bigraded Betti numbers
are defined for any simplicial complex. Nevertheless, in toric
topology the most important are simplicial complexes arising from
polytopes.

Let $P$ be a simple polytope with $m$ vertices. The polar dual
polytope $P^*$ is simplicial. The complex $K_P=\partial P^*$ is a
simplicial sphere with $m$ vertices. It is known \cite{BP2,BP}
that $\Z_{K_P}$ is a smooth compact manifold and the action of
$T^m$ on $\Z_{K_P}$ is smooth. The algebraic version of this fact
is Avramov--Golod theorem \cite[Th.3.4.4]{BPnew}. It states the
following. The $\Tor$-algebra $\Tor^{*,*}_{\ko[m]}(\ko[K];\ko)$ is
a (multigraded) Poincare duality algebra if and only if the
complex $K$ is Gorenstein*. Any simplicial sphere $K$ is
Gorenstein* \cite[Th.5.1]{Stan}. In particular, for any simple
polytope $P$ the complex $K_P$ is Gorenstein*, thus
$\Tor^{*,*}_{\ko[m]}(\ko[K_P];\ko)$ is a Poincare duality algebra.
This is not surprising since
$\Tor^{*,*}_{\ko[m]}(\ko[K_P];\ko)\cong H^*(\Z_{K_P};\ko)$ and
$\Z_{K_P}$ is a manifold.

The problems solved in this paper can be posed for particular
classes of simplicial complexes, for example boundaries of
simplicial polytopes or simplicial spheres.

\begin{problem}
Does $s(K_P)=\sr(K_P)$ for any simple polytope $P$?
\end{problem}

The complex $U=\Ur_4$ constructed in the proof of
Theorem~\ref{thmRealComplex} is not a boundary of a polytope; it
is not a simplicial sphere as well. Nevertheless, $U$ is
Cohen--Macaulay as proved in~\cite[Th.2.2]{DJ}.

Another problem can also be formulated for the class of polytopes.

\begin{problem}\label{probBettiBuch}
Does $\beta^{-i,2j}(K_P)=\beta^{-i,2j}(K_Q)$ imply $s(K_P)=s(K_Q)$
or $\sr(K_P)=\sr(K_Q)$ for simple polytopes $P$ and $Q$? If no,
does an isomorphism of algebras
$\Tor^{*,*}_{\ko[m]}(\ko[K_P];\ko)\cong
\Tor^{*,*}_{\ko[m]}(\ko[K_Q];\ko)$ imply $s(K_P)=s(K_Q)$ or
$\sr(K_P)=\sr(K_Q)$?
\end{problem}

The complexes $K_1$ and $K_2$ constructed in Section
\ref{SecBettiBuch} are not simplicial spheres as well. One can
deduce this from the table of bigraded Betti numbers (fig.
\ref{pictBettiSetCol}): if the complexes were spheres the
distribution of bigraded Betti numbers would be symmetric
according to (bigraded) Poincare duality.

It is tempting to modify the construction of $K_1$ and $K_2$ of
Section \ref{SecBettiBuch} to obtain polytopal spheres in the
output. Unfortunately, this attempt fails due to the following
observation.

\begin{prop}\label{propMinTayForPoly}
Let $K$ be a simplicial sphere. The Taylor resolution of $\ko[K]$
is minimal if and only if $K$ is a join of boundaries of
simplices.
\end{prop}

\begin{rem} For such $K$ it is easily shown that
$s(K)=\sr(K)=m~-~\dim~K~-~1$. Thus a counterexample to Problem
\ref{probBettiBuch} can not be constructed using minimal Taylor
resolutions.
\end{rem}

\begin{proof}[Proof of the proposition]
The ``if'' part is Example \ref{exProdSimpl}. Let us prove the
``only if'' part. Let $[m]$ be the vertex set of $K$. Any vertex
$i\in[m]$ is contained in at least one minimal non-simplex.
Otherwise, $K$ is a cone with the apex $i$, so $K$ is not a
sphere. Since the Taylor resolution is minimal, we may apply Lemma
\ref{lemMinTayTorStruc}. Complex $K$ is a sphere, thus $\ko[K]$ is
Gorenstein* and $\Tor^{*,*}_{\ko[m]}(\ko[K];\ko)$ is a multigraded
Poincare duality algebra. There should be a graded component of
$\Tor^{*,*}_{\ko[m]}(\ko[K];\ko)$ of maximal total degree which
plays the role of the ``fundamental cycle''. Obviously, this
component is generated by $W_{N(K)}$ in the notation of Lemma
\ref{lemMinTayTorStruc}. This component has multidegree
$(-|N(K)|,(2,2,\ldots,2))$. Non-degenerate pairing in Poincare
duality algebra $\Tor^{*,*}_{\ko[m]}(\ko[K];\ko)$ yields that for
each $\sigma\subseteq N(K)$ exists $\tau\subseteq N(K)$ such that
$W_{\sigma}\cdot W_{\tau} = \alpha W_{N(K)}$ with $\alpha\neq 0$.
Taking multigrading into account and applying Lemma
\ref{lemMinTayTorStruc} we get the following condition: for each
$\sigma\subseteq N(K)$ the vertex subsets $\bigcup_{J\in\sigma}J$
and $\bigcup_{J\in N(K)\setminus\sigma}J$ are disjoint. In
particular, any single non-simplex $J\in N(K)$ is disjoint from
the union of others. Therefore, $N(K)=\{J_1,\ldots,J_k\}$ and
$[m]=J_1\sqcup \ldots\sqcup J_k$. Thus
$K=(\partial\Delta_{J_1})\ast\ldots\ast(\partial\Delta_{J_k})$
which was to be proved.
\end{proof}

\section*{Acknowledgements}

I would like to thank Nickolai Erokhovets for useful discussions
and professor Xiangjun Wang from whom I knew the construction of a
Taylor resolution and its connection with moment-angle complexes.

\end{document}